\documentclass[a4paper,12pt]{article}
\usepackage[utf8]{inputenc}
\usepackage{geometry}
\usepackage[numbers]{natbib}
\usepackage{amsmath,amsthm,amsfonts,amssymb}
\usepackage{stmaryrd,booktabs,array,float,multirow}
\usepackage{graphicx,caption,subcaption,xcolor}
\usepackage{bm,textcomp}
\usepackage{enumitem}
\usepackage{hyperref}

\newtheorem{theorem}{Theorem}[section]
\newtheorem{lemma}[theorem]{Lemma}
\newtheorem{corollary}[theorem]{Corollary}
\newtheorem{definition}[theorem]{Definition}

\newtheorem{remark}{Remark}

\newcommand{\abs}[1]{\left\lvert#1 \right\rvert}
\newcommand{\norm}[1]{\lVert#1\rVert}

\newcommand\vr{\varrho}
\newcommand\vn{\bm{n}}
\newcommand\vu{\bm{u}}
\newcommand\vm{\bm{m}}
\newcommand\vvarphi{\bm{\varphi}}
\newcommand{\ve}{\bm{e}}

\newcommand{\Grad}{ \nabla}

\newcommand{\GradD}{ \nabla _{\cal D}}

\newcommand{\GradB}{ \nabla _{\cal B}}

\newcommand{\Td}{{\mathbb T}^d}
\newcommand{\faces}{\mathcal{E}}

\newcommand{\frakE}{ \mathfrak{E}}
\newcommand{\frakR}{ \mathfrak{R}}
\newcommand{\RE}{ {\cal R}_E}

\newcommand{\pd}{ \partial }
\newcommand\TS{\Delta t}

\newcommand\dt{\mathrm{dt}}
\newcommand\dx{\mathrm{d}x}
\newcommand\dS{\mathrm{dS}(x)}
\newcommand{\intO}[1]{\int_{\Omega} #1 \dx}
\newcommand{\intOB}[1]{\int_{\Omega} \left(#1 \right) \dx}

\newcommand{\intT}[1]{ \int_0^\tau #1 \dt}

\newcommand{\intE}[1]{\int_{\faces} #1 \dS}

\newcommand\aleq{\lesssim}

\newcommand{\Hc}{ {\cal H} }
\newcommand{\Qh}{Q_h}
\newcommand{\vQh}{{\bf Q}_h}
\newcommand{\vWh}{ {\bf W}_h}
\newcommand{\Whi}{W_{i,h}}

\newcommand{\pdt}{\pd_t}
\newcommand{\eps}{\varepsilon}

\newcommand\vrh{\vr_h}
\newcommand\vuh{\vu_h}

\newcommand\vmh{\vm_h}

\newcommand{\mesh}{\mathcal{T}}
\newcommand{\edges}{\mathcal{E}}

\newcommand\edgesi{\edges_i}

\newcommand{\edge}{\sigma}
\newcommand{\ha}{h^\eps}
\newcommand\Up{\mathrm{Up}}
\newcommand\Fup{F_\ha}

\newcommand{\ith}{ i^{\text{th}}}

\newcommand{\E}{\mathcal{E}}

\newcommand{\Di}{\mathcal{D}_i}

\newcommand{\Bij}{\mathcal{B}_{i,j}}

\newcommand{\Eij}{ \widetilde{\E}_{i,j}}
\newcommand\vx{x}
\newcommand{\bfx}{\vx} 
\newcommand\facei{\mathcal{E}_i}

\newcommand{\uih}{ u_{i,h} }
\newcommand{\ujh}{ u_{j,h} }

\newcommand{\auih}{ \overline{\uih} }

\newcommand{\sigmap}{\sigma _{_{K\text{,} i +}}}
\newcommand{\sigmam}{\sigma _{_{K\text{,} i -}}}

\newcommand{\pdedgei}{ \eth _{{\cal D}_i} }
\newcommand{\pdedgej}{ \eth _{{\cal D}_j} }

\newcommand{\pdduali}{\pdedgei}
\newcommand{\pddualj}{\pdedgej}

\newcommand{\pdmesh}[1]{ \pd _\mesh^{(#1)} }
\newcommand{\pdmeshi}{ \pdmesh{i} }
\newcommand{\pdmeshj}{ \pdmesh{j} }

\newcommand{\pdBij}{ \eth_{\Bij} }

\newcommand{\Div}{\nabla \cdot}

\newcommand\Divh{\mathrm{div}_h}

\newcommand\sumi{\sum_{i=1}^d}

\newcommand{\Piq}{\Pi _Q}
\newcommand{\Piv}{ \Pi _ \E}

\newcommand{\Pivi}{ \Pi _ \E^{(i)}}
\newcommand{\avuh}{\overline{\vuh}}

\newcommand{\jump}[1]{\left\llbracket#1\right\rrbracket}
\newcommand{\intK}[1]{\int_{K} #1 \dx}
\newcommand{\intSh}[1]{\int_{\sigma} #1 \dS}

\def\softd{{\leavevmode\setbox1=\hbox{d}%
          \hbox to 1.05\wd1{d\kern-0.3ex{\char039}\hss}}}

\makeatletter
\def\@cite#1#2{[#1]\if@tempswa\typeout
  {WSPC warning: optional citation argument ignored: `#2'}\fi}
\makeatother

\newcommand{\keywords}[1]{\par\noindent\textbf{Keywords: } #1\par}
\newcommand{\ccode}[1]{\par\noindent\textbf{AMS Subject Classification: }#1\par}

\makeatletter
\def\@cite#1#2{[#1]\if@tempswa\typeout
  {WSPC warning: optional citation argument ignored: `#2'}\fi}
\makeatother

\title{Convergence analysis of a MAC scheme for the barotropic Euler system}
\author{%
  Bangwei She\footnotemark[1] \and
  Congying Wang\footnotemark[2] \and 
  Jin Zhao\footnotemark[3]
}

\date{}
\begin{document}

\maketitle
\medskip
\footnotetext[1]{{Academy for Multidisciplinary Studies, Capital Normal University,\\
West 3rd Ring North Road 105, Beijing 100048, P. R. China.}\\
\texttt{bangweishe@cnu.edu.cn}}
\footnotetext[2]{School of Mathematical Sciences, Capital Normal University,\\
West 3rd Ring North Road 105, Beijing 100048, P. R. China.\\
\texttt{congyingw@163.com}}
\footnotetext[3]{{Academy for Multidisciplinary Studies, Capital Normal University,\\
West 3rd Ring North Road 105, Beijing 100048, P. R. China.}\\
\texttt{zjin@cnu.edu.cn}}

\begin{abstract}
We study a Marker-and-Cell (MAC) scheme for the barotropic Euler system. 
First, we apply the recently developed Lax-type convergence theorem to show the convergence of the MAC scheme to i) a dissipative weak solution unconditionally and ii) a strong solution as long as it exists. 
Second, We derive relative energy error estimates up to the lifespan of a strong solution, without assuming uniform boundedness of the numerical sequence. 
Additionally, assuming the boundedness of the numerical solutions, we obtain the optimal relative energy rate of 1, corresponding to a convergence rate of 1/2 for the numerical solutions. Finally, we corroborate our theoretical results by numerical experiments.
\end{abstract}


\ccode{65M12, 76M20}
\tableofcontents

\section{Introduction}	
\label{intro}
We consider the barotropic Euler system 
\begin{equation}\label{PDE}
\left\{
\begin{aligned}
& \partial_t \vr + \Div \vm = 0, \\
& \partial_t \vm + \Div \frac{\vm \otimes \vm}{\vr} + \Grad p(\vr) = \bm{0}
\end{aligned}
\right.
\end{equation}
in the time-space cylinder $[0,T] \times \Omega $. Here $\vr$, $\vm$, and $p(\vr)=\vr^\gamma$ are the density, momentum, and pressure of the fluid, respectively, and $\gamma>1$ is the adiabatic exponent. 
The system is complemented by the initial data
\begin{equation}\label{BC}
   \vr(0,\cdot) = \vr_0 >0, \ \vm(0,\cdot) =\vm_0, 
\end{equation}
and periodic boundary conditions, where the domain $\Omega$ is identified with the flat torus $$\Omega=\Td \equiv \left( [0,1] |_{\{ 0,1 \}} \right)^d. $$

Over the past few decades, many numerical methods have been successfully applied to the approximation of the Euler equations; see, e.g., Toro~\cite{Toro}, Feistauer et al.~\cite{FeisFel}, Ben-Artzi et al.~\cite{MAJ}, Godunov~\cite{Godunov}, Shu and Osher~\cite{ShuOsher}, and LeVeque~\cite{LeVeque}. 
However, a rigorous convergence analysis of numerical methods remains difficult in general. Concerning the full Euler system, Feireisl et al.~\cite{FLM_dmv} analyzed the convergence of the fundamental Lax--Friedrichs and Rusanov schemes. They established a convergence framework based on the dissipative measure-valued weak--strong uniqueness principle, valid on the lifespan of the strong solution. 
Moreover, using this convergence theory, Luk\'a\v{c}ov\'a-Medvi{\softd}ov\'a and Yuan demonstrated the convergence of the classical Godunov method~\cite{LMY}, which has also been extended to other methods, see e.g.~\cite{LYO,LYO2,KMP,DLT}. 
Concerning the barotropic Euler system,  Arun and Krishnamurthy~\cite{AAM2AN} have shown the convergence of a semi-implicit, entropy-stable finite volume scheme, under the boundedness of numerical solutions,  to a dissipative measure-valued solution. We also mention the vanishing-viscosity approach by Feireisl et al.~\cite{FLSS}. In the isentropic case, they introduced viscosity solutions and a viscosity finite volume method, and proved convergence of the numerical solutions to a dissipative measure-valued solution of the barotropic Euler system.

In this paper, we study a Marker-and-Cell (MAC) scheme for the barotropic Euler system. The MAC scheme was originally introduced by Harlow and Welch~\cite{HHJE} for the simulation of incompressible fluids. Herbin et al.~\cite{HKL} applied the MAC scheme to the Euler equations and later investigated its consistency in the spirit of the Lax--Wendroff theorem~\cite{HLMT}. This approach, however, assumes not only a priori boundedness but also the convergence of the numerical approximations. Here, our aim is to show the convergence of the MAC scheme without assumptions on the numerical sequences. To this end, we introduce a Navier-Stokes-type artificial diffusion to the MAC scheme, which provides an additional weak BV estimate of the velocity (negative $L^2$ estimate of the velocity gradient). This is the key ingredient in proving the unconditional consistency of the MAC scheme, see Lemma~\ref{MAC_CS}. 
We then apply the generalized Lax equivalence theory~\cite{FLMS_book}, and obtain our primary result: the convergence of the MAC scheme, as detailed in Theorem~\ref{theo_Con}.

Another aim of this paper is to investigate the convergence rate of the MAC scheme for the barotropic Euler system. Luk\'a\v{c}ov\'a-Medvi{\softd}ov\'a et al.~\cite{LMSY1} analyzed the error estimates of the Godunov method for the full Euler system. They obtained a convergence rate of 1/4 in $L^2$-norm, assuming the boundedness of numerical density and energy. Moreover, under the additional assumption that the total variation of the numerical solutions is bounded, they achieved an improved convergence rate of 1/2 in $L^2$-norm. 
For multidimensional nonlinear systems, Jovanovi\'c and Rohde~\cite{JovaRoh} obtained a convergence rate of 1/2 in $L^2$-norm, under the assumptions that the numerical solutions are uniformly bounded and the first-order derivatives are bounded in the $L^2$ and $L^\infty$ norms. Thanks to the introduction of the Navier--Stokes type artificial diffusion, we obtain the second result: the same convergence rate of 1/2, but only under the assumption that the numerical solutions are bounded, see Theorem~\ref{theo_EE}.

We highlight the main results of this paper:
\begin{itemize}
\item We prove the unconditional convergence of the MAC scheme to a dissipative weak solution; moreover, if a strong solution exists,
the numerical solutions converge to this strong solution (see Theorem~\ref{theo_Con}). The local existence of strong solutions is ensured by~\cite[Theorem~13.1]{SDCP}.

\item We derive error estimates without assuming uniform boundedness of the numerical solutions. Under the additional assumption that the numerical solutions are bounded,
we obtain the optimal $1/2$-order convergence rate for the numerical solutions (see Theorem~\ref{theo_EE}), under weaker assumptions than those typically used in the literature.
\end{itemize}

The rest of this paper is organized as follows.
In Section~\ref{euler}, we collect the analytical preliminaries.
In Section~\ref{Con_MAC}, we introduce the MAC scheme and establish its stability, consistency, and convergence.
In Section~\ref{error_es}, we derive error estimates for the MAC scheme.
Finally, in Section~\ref{numerical_ex}, we present numerical experiments illustrating the behavior of the MAC scheme.

\section{Solution concepts and consistent approximations}\label{euler}
In this section, we recall the notions of dissipative weak and strong solutions for the barotropic Euler system and introduce the concept of a consistent approximation. These notions provide the analytical framework for the convergence analysis of the MAC scheme.

\subsection{Solution concepts}
\begin{definition}[\bf Dissipative weak solution]\label{defi_DWs}
Let $\Omega = \Td \subset \mathbb{R}^d, d=2,3$. We say that $(\vr, \vm)$ is a dissipative weak (DW) solution of the barotropic Euler system \eqref{PDE} -- \eqref{BC} if the following hold:
\begin{subequations}
\begin{itemize}
    \item {\bf Weak continuity.}
    \begin{equation}
        \vr \in C_{weak}([0, T]; L^{\gamma}(\Omega)), \quad \vm \in C_{weak}([0, T]; L^{\frac{2\gamma}{\gamma+1}}(\Omega; \mathbb{R}^d));
    \end{equation}
    \item {\bf Energy stability.} There is a defect measure $\frakE \in L^{\infty}(0,T; \mathcal{M}^+({\Omega}))$ such that the energy inequality holds
    \begin{equation}\label{ener_ieq}
        \intO{\left(\frac{|\vm|^2}{2\vr} + \mathcal{H}(\vr)\right)(\tau, \cdot)} + \int_{{\Omega}} \mathrm{d}\frakE(\tau) \leq \intOB{\frac{|\vm_0|^2}{2\vr_0} + \mathcal{H}(\vr_0)}
    \end{equation}
   for any $\ 0 \leq \tau \leq T$, where $\Hc(\vr) = \frac{p(\vr)}{\gamma-1}$ is the  pressure potential associated to the state equation $p(\vr)=\vr^\gamma$;
    \item {\bf Continuity equation.}
    
    \begin{equation}\label{dw3}
        \left[ \intO{\vr \varphi} \right]^{t=\tau}_{t=0} \mbox{ = } \intT{\intOB{\vr \partial_t \varphi + \vm \cdot \Grad\varphi}}
    \end{equation}
    for any $0 \leq \tau \leq T$ and any $\varphi \in W^{1,\infty}((0,T) \times \Omega)$;
    \item {\bf Momentum equation.}
    
    \begin{equation}\label{dw4}
    \begin{aligned}
    \left[ \intO{\vm \cdot \vvarphi} \right]^{t=\tau}_{t=0} 
    & \mbox{ = } \intT{\intOB{\vm \cdot \partial_t\vvarphi + 1_{\vr > 0}\frac{\vm \otimes \vm}{\vr} : \Grad\vvarphi + p(\vr)\nabla\cdot\vvarphi }}\\
    & +\intT{\int_{{\Omega}} \Grad\vvarphi: \mathrm{d}\frakR(t)}
    \end{aligned}
    \end{equation}
     for any $0 \leq \tau \leq T$ and any $\vvarphi \in C^M([0,T] \times {\Omega}; \mathbb{R}^d), M \geq 1$ with the Reynolds defect
     $$\frakR \in L^\infty(0,T;\mathcal{M}^+({\Omega}; R_{sym}^{d\times d}));$$
     \item {\bf Defect compatibility condition.}
     \begin{equation}\label{dw5}
         \underline{d} \frakE \leq tr [\frakR] \leq \overline{d} \frakE \mbox{  for some constants } 0 \leq \underline{d} \leq \overline{d}.
     \end{equation}
\end{itemize}
\end{subequations}
\end{definition}

\begin{definition}[\bf Strong solution]\label{defi_Cs}
Let $\Omega =\Td \subset \mathbb{R}^d, d=2,3$. We say that $(\vr, \vu)$ is a strong solution of the barotropic Euler system \eqref{PDE} -- \eqref{BC} if:
\begin{equation*}
\begin{aligned}
&\vr \in C^1([0, T]\times{\Omega}),
&\vu \in C^1([0, T]\times{\Omega};\mathbb{R}^d) 
\end{aligned}
\end{equation*}
\begin{equation*}
0<\underline{\vr}\leq\vr(t, x) \mbox{ for any }t\in[0,T], x\in\Omega;
\end{equation*}
and the equations \eqref{PDE} hold.
\end{definition}
For the local existence of strong solutions, we refer the reader to Benzoni--Gavage and Serre~\cite[Theorem~13.1]{SDCP}.

\subsection{Consistent approximations}
\begin{definition}[\bf Consistent approximation]\label{defi_CA}
Let $(\vrh^0,\vmh^0)$ be an approximation of the initial data $(\vr_0,\vm_0)$ such that
\[
\vrh^0 \to \vr_0 \quad \text{weakly in } L^1(\Omega),
\qquad
\vmh^0 \to \vm_0 \quad \text{weakly in } L^1(\Omega;\mathbb{R}^d)
\quad \text{as } h\to 0.
\]
We say that a numerical approximation $(\vrh,\vmh)$ is a consistent approximation of the barotropic Euler system \eqref{PDE}--\eqref{BC} if the following stability and consistency conditions hold:
\noindent{\bf i) The stability conditions.} 
\begin{subequations}
\begin{itemize}
    \item {\bf Positivity of density.}
    \begin{equation}\label{sta1}
        \vrh(t)>0, \quad t\in [0,T];
    \end{equation}
    \item {\bf Energy stability.}
    \begin{equation}\label{sta2}
       \intO{\left( \frac{|\vmh|^2}{2\vrh} + \Hc(\vrh)\right)(t,\cdot)}  \leq \intOB{\frac{|\vmh^0|^2}{2\vrh^0} + \Hc(\vrh^0)},  \quad t\in[0,T].
    \end{equation}
\end{itemize}

\noindent{\bf ii) The consistency conditions.} 
\begin{itemize}
    \item {\bf Continuity equation.} It holds for any $\tau\in(0,T),\ \varphi \in C_c^2([0,T) \times \Omega)$ that

    \begin{equation}\label{cons1}
\left[\intO{ \vrh \varphi }\right]_{t=0}^{\tau} =
 \int_0^{\tau} \intOB{ \vrh \partial_t \varphi + \vmh \cdot \Grad \varphi} \dt  + 
R_{\vr} ({\tau},\Delta t,h, \varphi),
    \end{equation}
where $R_{\vr} ({\tau},\Delta t,h, \varphi) \rightarrow 0 $ as $h\rightarrow 0$;
    \item {\bf Momentum equation.} It holds for any $\tau\in(0,T),\ \vvarphi \in C^2_c([0,T) \times \Omega; \mathbb{R}^d)$ that
    \begin{multline}\label{cons2}
    \left[\intO{ \vmh \cdot \vvarphi}\right]_{t=0}^{\tau} \\ 
= \int_0^{\tau} \intOB{\vmh \cdot \partial_t \vvarphi + \left(\frac{\vmh \otimes \vmh}{\vrh} \right) : \Grad \vvarphi  + p(\vrh) \Div \vvarphi} \dt + R_{\vm} ({\tau},\Delta t,h, \vvarphi),
    \end{multline}
where $R_{\vm} ({\tau},\Delta t,h, \vvarphi) \rightarrow 0 $ as $h\rightarrow 0$.
\end{itemize}
\end{subequations}
\end{definition}


\section{The MAC scheme and convergence}\label{Con_MAC}
In this section, we introduce a MAC scheme and show its convergence to i) a dissipative weak solution without any assumptions and ii) the strong solution as long as the latter exists. 

\subsection{The MAC scheme}
In this subsection, we introduce the MAC scheme for the Euler system. To begin with, we introduce the mesh and the necessary notation. 
Let $\mesh$ be a uniform structured mesh of $\Omega = \cup_{K\in \mesh}K$, $\edges$ be the set of all faces of $\mesh$, and $\edgesi$ be the set of all faces orthogonal to the unit basis vector $\ve_i$. We write $\edge = K|L$ as the common face of neighboring elements $K$ and $L$. For any $\edge = K|L$, a dual cell $D_{\edge} = D_{K, \edge}\cup D_{L, \edge}$ is defined, where $D_{K, \edge}$ and $D_{L, \edge}$ are the half-cells of $K$ and $L$ near $\edge$, respectively. The set $D_i := \{D_{\edge}|\edge\in \edgesi \}$ represents the $\ith$ dual grid. Similarly, a bidual cell $D_\epsilon := D_{\epsilon,\sigma} \cup D_{\epsilon,\sigma'}$, associated with $\epsilon = D_\sigma | D_{\sigma'}$ (denoting the interface of neighboring cells $D_\sigma$ and $D_{\sigma'}$), is defined as the union of adjacent halves of the cell
$D_\sigma$ and $D_{\sigma'}$ near $\epsilon$. Further, 
$\Eij = \{ \epsilon \in \widetilde{\E}_i | 
\epsilon \mbox{ is orthogonal to } \ve_j\}$, where $\widetilde{\E}_i$ is the set of all faces of the $\ith$ dual grid $\Di$.
Finally, $\sigmam$ and $\sigmap$ denote the left and right faces of an element $K$ in the $\ith$ direction, respectively. 

With the above notation, we define discrete spaces $\Qh$ and $\Whi$ of piecewise constants on the primary mesh $\mesh$ and the $\ith$ dual mesh $\Di$, respectively. Moreover, we set $\vQh =\Qh^{d}$ and $\vWh  =\left(W_{1,h},\ldots W_{d,h} \right)$.

Further, we define the following discrete operators:
\begin{equation*}
D_t f = \frac{ f(t,\cdot) - f(t-\TS,\cdot)}{\TS},\quad
\Piq \phi= \sum_{K\in \mesh} (\Piq \phi)_K  1_{K},\quad (\Piq \phi)_K =  \frac{1}{|K|} \int_{K} \phi \dx,
\end{equation*}
\begin{equation*}
 \auih|_K = \frac{ \uih |_{\sigmap}+ \uih|_{\sigmam} }{2}, \quad
\auih  = \sum_{K \in \mesh} 1_{K}  \auih|_K,   \quad 
\overline{\vuh}  = \left( \overline{u_{1,h}},\ldots, \overline{u_{d,h}} \right),
\end{equation*}
\begin{equation*}
 \GradB \vuh (\bfx) = \big( \GradB u_{1,h}, \ldots, \GradB u_{d,h}\big)(\bfx),
    \GradB \uih (\bfx) = \big( \eth_{{\cal B}_{i,1}} \uih, \ldots, \eth_{{\cal B}_{i,d}} \uih \big)(\bfx),
\end{equation*}
\[
\left. \pdBij \uih\right|_{D_\epsilon} = \frac{u_{\sigma'} -u_{\sigma}}{h}, \mbox{ for } \epsilon =  D_\sigma | D_{\sigma'} \in \Eij. 
\]
\begin{equation*}
\Divh \vuh(\bfx)  = \sumi  \pdmeshi \uih (\bfx) ,\quad 
 \left. \pdmeshi \uih \right \vert_{K} = \frac{ \uih|_{\sigmap} - \uih|_{\sigmam}}{h},\quad K\in\mesh,
\end{equation*}
\begin{equation*}
     \Fup  [r_h,\vuh]_\sigma =  \Up[r_h,\vuh]_\sigma - \ha \jump{r_h}_\sigma, \quad  \eps > -1,
\end{equation*}
\begin{equation*}
\Up [r_h,\vuh]_\sigma 
= r_h^{\rm in} (u_{\sigma})^+ + r_h^{\rm out} (u_{\sigma})^-,\quad \jump{r_h}_{\sigma}(x) = r_h^{\rm out}(x) -  r_h^{\rm in}(x),
\end{equation*}
$$\quad r_h^{\rm out}(x) = \lim_{\delta  \to 0+} r_h(x+\delta \vn) \quad \ r_h^{\rm in}(x) = \lim_{\delta  \to 0+} r_h(x - \delta \vn),$$
\begin{equation*}
\begin{aligned}
r_h^{\rm up} =
\begin{cases}
r_h^{\rm in}  &\mbox{ if } u_\sigma \geq 0,\\
r_h^{\rm out} &\mbox{ if } u_\sigma < 0,
\end{cases}
\quad
r^{\pm} = \frac{1}{2} (r \pm |r|), \quad
u_\sigma= 
\vuh|_\sigma \cdot \vn_\sigma,
\end{aligned}
\end{equation*}
where $\vn_\sigma$ is the outer normal vector to $\sigma$.
For a vector-valued function $\bm{w}_h=(w_{1,h},\ldots,w_{d,h})$
$ \in\vQh$, 
$\Fup[\bm{w}_h,\vuh]=\big(\Fup[w_{1,h},\vuh],\ldots,\Fup[w_{d,h},\vuh]\big).$
Throughout the paper, we use the shorthand notations:
$$\|f\|_{L^qL^p}:=\|f\|_{L^q(0,T;L^p(\Omega))},\quad \|f\|_{L^p}:=\|f\|_{L^p(\Omega)}.$$

With the above notations, we define the MAC scheme.  
\begin{definition}[\bf MAC scheme]\label{defMAC}
Given the initial data $(\vr_0, \vm_0)$, we set  $(\vrh^0,\vmh^0 ) \allowbreak=(\Piq\vr_0, \Piq\vm_0).$
The MAC approximation of the Euler system \eqref{PDE} -- \eqref{BC}  is a sequence
$ (\vrh , \vuh  ) \in \Qh \times \vWh,$
which solves  the following system of algebraic equations:
\begin{subequations}\label{MAC_S}
\begin{equation}\label{MAC_SD}
\intO{ D_t \vrh  \phi_h } -   \intE{  \Fup [\vrh ,\vuh ]\jump{\phi_h}   }  = 0  \mbox{ for all }\  \phi_h \in \Qh,
\end{equation}
\begin{multline} \label{MAC_SM}
\intO{ D_t \vmh \cdot \overline{\Phi_h} } -   \intE{ \Fup [\vmh,\vuh ] \cdot \jump{\overline{\Phi_h} }   }
  - \intO{  p(\vrh)    \Divh \Phi_h   }
   \\ 
+ h^\alpha \intO{ \GradB \vuh:  \GradB \Phi_h  } 
=0
\mbox{ for all  } \Phi_h \in \vWh.
\end{multline}
\end{subequations}
\end{definition}

Here, 
$\vmh = \vrh \avuh$. The last term on the left-hand side of \eqref{MAC_SM} represents a Navier--Stokes-type artificial diffusion term with $ \alpha>0$.

\subsection{Stability, consistency, and convergence}
In this subsection, we analyze the stability, consistency, and convergence of the MAC scheme \eqref{MAC_S}, and show that it admits at least one solution.
\begin{lemma}[\bf Stability of MAC scheme]\label{MAC_sta}
Let $(\vrh,\vuh)$ be any solution of the MAC scheme \eqref{MAC_S}. Then the following properties hold.
\begin{subequations}
\begin{itemize}
    \item {\bf Positivity of density.}    
    Let $\vr_0>0$. Then any solution of the MAC scheme \eqref{MAC_S} satisfies $\vrh(t)>0$ for any $t\in [0,T]$;
    \item{\bf Renormalized continuity equation.}
    Let $(\vrh,\vuh)$ be a solution of the discrete continuity equation
\eqref{MAC_SD}, and let $B\in C^2(0,\infty)$. Then the discrete continuity equation \eqref{MAC_SD} can be renormalized in the sense that:
\begin{equation}\label{rce}
\begin{aligned}
&\intOB{D_t B(\vrh)+\left( \vrh B^\prime(\vrh) - B(\vrh) \right)\Divh \vuh} \\
& =-\frac{\Delta t}{2}\intO{B^{\prime\prime}(\xi)|D_t \vrh|^2} - \sum_{\edge\in\faces}\int_\edge B^{\prime\prime}(\varsigma)\jump{\vrh}^2\left( h^\eps + \frac{1}{2}\left| u_\sigma \right| \right) \dS,
\end{aligned}
\end{equation}
where $\xi \in co\{ \vrh(t-\Delta t), \vrh(t) \}$, $\varsigma \in co\{ \vrh^{in},\vrh^{out} \}, co\{a,b\}:=[\min\{a,b\},\allowbreak\max\{a,b\}]$;

    \item {\bf Energy stability.} Let $(\vrh, \vuh)$ be a numerical solution obtained from the MAC scheme \eqref{MAC_S}. Then there exist $\xi \in co\{ \vrh(t-\Delta t), \vrh(t) \}$ and $\varsigma \in co\{ \vrh^{in},\vrh^{out} \}$ such that
\begin{align}\label{Ener_Sta}
&D_t\intO{\left(\frac{1}{2}\vrh|\overline{\vuh}|^2 + \mathcal{H}(\vrh)\right)} + h^\alpha\intO{|\GradB\vuh|^2} \notag \\
&= -\frac{\Delta t}{2}\intO{\mathcal{H}^{\prime\prime}(\xi)| D_t\vrh |^2} - \int_\mathcal{E} \mathcal{H}^{\prime\prime}(\varsigma)\jump{\vrh}^2 \left( h^\eps + \frac{1}{2} \left| u_\sigma \right| \right) \dS \notag \\
&\quad -\frac{\Delta t}{2}\intO{\vrh(t-\Delta t)| D_t\overline{\vuh}|^2} - \frac{1}{2} \int_\mathcal{E} \vrh^{\rm up} \left| u_\sigma \right|\jump{\overline{\vuh}}^2 \dS.
\end{align}
\end{itemize}
\end{subequations}
\end{lemma}
\begin{proof}
\begin{itemize}
    \item The positivity of the density and the renormalized form of the discrete continuity equation follow from the upwind structure of \eqref{MAC_SD}; see \cite[Lemmas 8.2 and 8.3]{FLMS_book}.
    \item 
We achieve energy stability by combining \eqref{MAC_SD}, \eqref{MAC_SM} and the renormalized continuity equation \eqref{rce} with $\phi_h = -\frac{1}{2}|\overline{\vuh}|^2$, $\Phi_h= \vuh$, and $B(\vrh) =\mathcal{H}(\vrh)$. Using the identity $\vr\mathcal H'(\vr)-\mathcal H(\vr)=p(\vr)$, we obtain \eqref{Ener_Sta}.
\end{itemize}
\end{proof}

By adapting the topological degree argument in~\cite[Lemma 11.3]{FLMS_book}, we obtain the existence of a numerical solution to the present MAC scheme.
\begin{lemma}[\bf Existence of a numerical solution]\label{lem:existence_MAC}
Assume that the initial density $0<\underline{\vr}\le \vr_0 \quad \mbox{in }\Omega$.
Let $\vr_h^0=\Pi_Q\vr_0$. Then, for any $k=1,\dots,N_T$, $t^{N_T}=T$, there exists at least one solution
$(\vr_h^k,\vu_h^k)\in\Qh\times\vWh$ to the MAC scheme
\eqref{MAC_S}.
\end{lemma}
\begin{proof}
The proof follows the topological degree argument in
\cite[Lemma 11.3]{FLMS_book}. In the present MAC setting, the only changes
are that the velocity belongs to the staggered space $\vWh$, the momentum is
the cell-centered quantity $\vm_h=\vr_h\overline{\vuh}$, and the viscous
regularization is replaced by the artificial diffusion term
$h^\alpha(\GradB\vuh,\GradB\Phi_h)$.

For the homotopy used in \cite[Lemma 11.3]{FLMS_book}, the case $\zeta=0$
reduces to a linear coercive problem on $\vWh$. The a priori bound required
for the degree argument is provided by the stability estimate in
Lemma~\ref{MAC_sta}, while the positivity of the density follows from the
same lemma. Hence no homotopy solution reaches the boundary of the
admissible set. The homotopy invariance of the topological degree gives a
solution for $\zeta=1$. Induction over the time levels completes the proof.
\end{proof}

Combining the energy inequality~\eqref{Ener_Sta}, the Sobolev--Poincar\'e inequality~\cite[Theorem~17]{FLMS_book}, and the negative norm estimates established in~\cite[Lemma~4.4]{LMSY} and~\cite[Corollary~11.1]{FLMS_book}, we obtain the following corollary.
\begin{corollary}\label{cor_unib}
$\mathbf{(Uniform\ bounds\ and\ negative\ estimates).}$
Let $(\vrh, \vuh)$ be a solution of the MAC scheme~\eqref{MAC_S}. Then there exists $c>0$, depending on the initial energy $E_0$ but independent of the discretization parameters $\TS$ and $h$, such that
\begin{subequations}\label{unib}
\begin{equation}\label{unib1}
\norm{\vrh|\overline{\vuh}|^2}_{L^\infty L^1} \leq c,\qquad
\norm{\vrh}_{L^\infty L^\gamma}\leq c,\qquad
\norm{\vmh}_{L^\infty L^{\frac{2\gamma}{\gamma+1}}}\leq c,
\end{equation}
\begin{equation}\label{unib2}
\norm{\vuh}_{L^2 L^6}\leq\norm{\GradB\vuh}_{L^2 L^2} \leq c h^{-{\alpha}/{2}}.
\end{equation}
\begin{equation}\label{nedm1}
\begin{aligned}
\norm{\vrh}_{L^2 L^2}\aleq h^{\beta_D},\qquad
\beta_D =
\begin{cases}
\min\limits_{p\in\left[1,\infty\right)}\{\frac{p(\eps+1)+4}{2p}, 1\}\cdot \frac{\gamma-2}{\gamma}
& \mbox{ if }d=2,\ \gamma \in(1,2), \\
\min\{\frac{\eps+2}{3}, 1\}\cdot\frac{3(\gamma-2)}{2\gamma}
&\mbox{ if }d=3,\ \gamma\in(1,2), \\
0 &\mbox{ if } \gamma \geq 2;
\end{cases}
\end{aligned}
\end{equation}
\begin{equation}\label{nedm3}
\begin{aligned}
\norm{\vrh\overline{\vuh}}_{L^2 L^2}\aleq h^{\beta_M},\qquad
\beta_M = \max\{\beta_{M_1}, -\frac{3\eps+3+d}{3\gamma}\},
\end{aligned}
\end{equation}
where
\begin{equation*}
\begin{aligned}
\beta_{M_1}=
\begin{cases}
\max\limits_{p\in\left[\frac{2\gamma}{\gamma-1},\infty\right)}\left\{ -\frac{p(\eps+1)+4}{2p\gamma},\frac{p(\gamma-2)-2\gamma}{p\gamma}-\frac{\alpha}{2} \right\}
& \mbox{ if }d=2,\ \gamma \in\left(1,2\right], \\
-\frac{\alpha}{2} &\mbox{ if }d=2,\ \gamma > 2,\\
\max\left\{ -\frac{\eps+2}{2\gamma},\frac{\gamma-3}{\gamma}-\frac{\alpha}{2}, -\frac{3}{2\gamma} \right\}
&\mbox{ if }d=3,\ \gamma\in\left(1,2\right],\\
\frac{\gamma-3}{\gamma}-\frac{\alpha}{2} &\mbox{ if }d=3,\ \gamma \in (2,3), \\
-\frac{\alpha}{2} &\mbox{ if }d=3,\ \gamma \geq 3.
\end{cases}
\end{aligned}
\end{equation*}
\end{subequations}
\end{corollary}
Here $ \aleq $ indicates that one quantity is less than or equal to another, up to a positive constant factor. 
\begin{remark}\label{va_betaM}
For the reader's convenience, we list below sufficient conditions ensuring $ \beta_D\ge\beta_M>-1+\frac{\alpha}{2}.$

\begin{itemize}
  \item \textbf{Case $d=2$.}
  \begin{enumerate}[label=(\roman*)]
\item If $\gamma \in (1,2]$, then
\begin{equation*}
\begin{aligned}
\begin{cases}
\varepsilon>-1,
&
0<\alpha<2-\dfrac{2}{\gamma},
\\ 
-1<\varepsilon<\gamma(2-\alpha)-1,
&
2-\dfrac{2}{\gamma}\le \alpha<2.
\end{cases}
\end{aligned}
\end{equation*}
\item If $\gamma > 2$, then
\begin{equation*}
\begin{aligned}
\begin{cases}
\varepsilon>-1,
&
0<\alpha<1,
\\ 
-1<\varepsilon<
\gamma\left(1-\dfrac{\alpha}{2}\right)-\dfrac53,
&
1\le \alpha<2-\dfrac{4}{3\gamma}.
\end{cases}
\end{aligned}
\end{equation*}
\end{enumerate}
    
  \item \textbf{Case $d=3$.}
 \begin{enumerate}[label=(\roman*)]
\item If $\gamma \in (1,2]$, then
      \begin{equation*}
\begin{aligned}
\begin{cases}
\varepsilon>-1,
&
0<\alpha<2-\dfrac{3}{\gamma},
\\
-1<\varepsilon<\gamma(2-\alpha)-2,
&
\max\left\{0,2-\dfrac{3}{\gamma}\right\}
\le \alpha<2-\dfrac1\gamma.
\end{cases}
\end{aligned}
\end{equation*}
      \item If $\gamma \in (2,3)$, then
      \begin{equation*}
\begin{aligned}
\begin{cases}
\varepsilon>-1,
&
0<\alpha<2-\dfrac{3}{\gamma},
\\
-1<\varepsilon<
\gamma\left(1-\dfrac{\alpha}{2}\right)-2,
&
2-\dfrac{3}{\gamma}\le \alpha<2-\dfrac{2}{\gamma}.
\end{cases}
\end{aligned}
\end{equation*}
      \item If $\gamma \ge 3$, and 
      \begin{equation*}
\begin{aligned}
\begin{cases}
\varepsilon>-1,
&
0<\alpha<1,
\\
-1<\varepsilon<
\gamma\left(1-\dfrac{\alpha}{2}\right)-2,
&
1\le \alpha<2-\dfrac{2}{\gamma}.
\end{cases}
\end{aligned}
\end{equation*}
\end{enumerate}
\end{itemize}
\end{remark}

With the above estimates, we present the consistency of the scheme.
\begin{lemma}[\bf Consistency of the MAC scheme]\label{MAC_CS}
Let $(\vrh, \vuh)$ be a solution of the MAC scheme \eqref{MAC_S}  with $\Delta t \approx h \in(0,1)$, $\gamma>1$.
In addition, assume that $\varepsilon$ and $\alpha$ satisfy the conditions 
  of Remark~\ref{va_betaM}. 
Then, for any $\tau\in(0,T),\ \varphi \in C_c^2([0, T) \times \Omega)$ and $\vvarphi \in C_c^2([0, T) \times \Omega; \mathbb{R}^d)$ there hold
\begin{subequations}\label{CS}
\begin{equation} \label{cs1}
\begin{aligned}
\left[\intO{ \vrh \varphi }\right]_{t=0}^{\tau}
 &= \int_0^{\tau} \intO{\left( \vrh\partial_t\varphi + \vmh \cdot \Grad\varphi \right)} \dt + R_{\vr} ({\tau},\Delta t,h, \varphi),
\end{aligned}
\end{equation}
\begin{equation}\label{cs2}
\begin{aligned}
\left[\intO{ \vmh \cdot \vvarphi}\right]_{t=0}^{\tau}&= \int_0^{\tau}\intO{\left( \vmh \cdot \partial_t \vvarphi + \frac{\vmh\otimes\vmh}{\vrh}:\Grad\vvarphi+ p(\vrh) \nabla\cdot\vvarphi\right)} \dt \\
&+R_{\vm} ({\tau},\Delta t,h, \vvarphi),  
\end{aligned}
\end{equation}
\end{subequations}
where the consistency errors are bounded as follows:
\begin{equation*}
\begin{aligned}
&\abs{R_{\vr} ({\tau},\Delta t,h, \varphi)} \aleq \Delta t+h^{\eps+1} +h  +h^{1+\beta_D-{\alpha}/{2}},\\
&\abs{R_{\vm} ({\tau},\Delta t,h, \vvarphi)} \aleq \Delta t+h^{\eps+1}+h +h^{1+\beta_M-{\alpha}/{2}}+h^{{\alpha}/{2}}.
\end{aligned}
\end{equation*}
\end{lemma}
The proof of this lemma is based on H\"older's inequality, Young's inequality, and the standard interpolation error estimates, see analogous result of Navier--Stokes equation~\cite{FLS_IEE}. 
For better readability, we provide a sketch of the proof in Appendix~\ref{App1}.
\begin{remark}
This lemma establishes the unconditional consistency of the MAC scheme. Without the weak BV estimate~\eqref{unib2} obtained from the artificial diffusion, consistency can only be proved under additional a priori assumptions, see~\cite{AAM2AN}.
\end{remark}

Now we can show the convergence of our scheme.
\begin{theorem}[\bf Convergence of MAC scheme]\label{theo_Con}
    Let $(\vrh, \vuh)$ be a numerical solution of the MAC scheme \eqref{MAC_S} with $\Delta t \approx h \in(0,1)$, $\gamma>1$. Assume further that the scheme parameters $(\varepsilon,\alpha)$ satisfy the requirements of Remark \ref{va_betaM}.

\begin{itemize}
 
\item There exists a subsequence (not relabelled), such that
\begin{equation*}
\begin{aligned}
&\vr_h \to \vr \quad \text{weakly-(*) in } L^\infty(0,T;L^\gamma(\Omega)),
\\
&\vm_h \to \vm \quad \text{weakly-(*) in } L^\infty\!\left(0,T;L^{\frac{2\gamma}{\gamma+1}}(\Omega;\mathbb{R}^d)\right),
\end{aligned}
  \end{equation*}
  where $(\vr,\vm)$ is a DW solution of the barotropic Euler system in the sense of Definition~\ref{defi_DWs};
  
  \item
Let $(\vr,\vu)$ be a strong solution of \eqref{PDE}--\eqref{BC} on $[0,T]$ in the sense of Definition~\ref{defi_Cs}.
Then, for any $1\le q<\infty$,
\begin{equation*}\label{scon}
\begin{aligned}
&\vr_h \to \vr \ \text{in } L^q(0,T;L^\gamma(\Omega)),
\\
&\vm_h \to \vr \vu \ \text{in }
L^q\!\left(0,T;L^{\frac{2\gamma}{\gamma+1}}(\Omega;\mathbb{R}^d)\right).
\end{aligned}
\end{equation*}

\end{itemize}    
\end{theorem}

\begin{proof}
Lemma~\ref{MAC_sta} (together with Corollary~\ref{cor_unib}) yields the stability requirements in Definition~\ref{defi_CA},
while Lemma~\ref{MAC_CS} (under the parameter constraints of Remark~\ref{va_betaM}) provides the consistency conditions.
Hence $(\vrh,\vmh)$ is a consistent approximation. Therefore, by \cite[Theorems~5.3 and~5.4]{FLMS_book},
the numerical solutions converge to a dissipative weak solution; moreover, if a strong solution exists,
they converge to that strong solution.
\end{proof}

\begin{remark}
There is another general convergence theory by Lax and Wendroff in the context of a hyperbolic system of conservation laws. It states that if the solution of a conservative numerical scheme is bounded and converges strongly, then the limit is a weak solution. This theory has been extended to the Euler system by Herbin et al. \cite{HLMT}. 
Let us point out that our convergence result stated in the second part of Theorem \ref{theo_Con} only requires the existence of the strong solution. Without knowing the existence of a strong solution we can still claim the natural convergence towards a DW solution, see the first part of Theorem \ref{theo_Con}. Note that this theory does not need any assumption on the approximate sequence. Conversely, the Lax-Wendroff theory requires a priori both the boundedness and strong convergence of the approximate sequence, see  \cite{HLMT}. 
\end{remark}

\section{Error estimates}\label{error_es}
In this section, we study the convergence rate of the MAC scheme \eqref{MAC_S} following~\cite{FLS_IEE}. To this end, we introduce the relative energy functional
\begin{align*}
    \RE\{(\vrh,\vuh)|(\vr,\vu)\}= \intOB{\frac12 \vrh |\overline{\vuh}-\vu|^2 +\Hc(\vrh) - \Hc(\vr)-\Hc'(\vr)(\vrh-\vr)}
\end{align*}
that represents the distance between numerical solution $(\vrh,\vuh)$ and target solution $(\vr,\vu)$. 
\begin{theorem}[\bf Error estimates]\label{theo_EE}
Let the Euler system \eqref{PDE} - \eqref{BC} admit a strong solution $(\vr,\vu)$ in the class of $W^{2,\infty}((0,T)\times\Omega)$ with $\vr>0$ and $\gamma>1$. Let $(\vrh, \vuh)$ be a solution of the MAC scheme \eqref{MAC_S}. 
\begin{itemize}
\item 
Then
\begin{align}\label{ineq_error}
\sup_{0\leq t \leq T} \RE\{(\vrh,\vuh)|(\vr,\vu)\} \aleq  \TS + h^A, 
\end{align}
where the convergence rate $A$ reads
\begin{align}\label{uncon_e}
A=\min\{1, 1+\eps, \frac{\alpha}{2}, 1+\beta_D -\frac{\alpha}{2}, 1+\beta_M -\frac{\alpha}{2}\}.
\end{align}
Here, the constants $\beta_D$ and $\beta_M$ are given in Corollary \ref{cor_unib}.
Moreover,

\begin{align}\label{uncon_b1}
&\norm{\vrh-\vr}_{L^\gamma}+\norm{\vmh-\vm}_{L^\frac{2\gamma}{\gamma+1}}\aleq (\TS + h^A)^\frac{1}{\gamma} \mbox{ for }\gamma\leq 2,\notag \\
&\norm{\vrh-\vr}_{L^2}+\norm{\vmh-\vm}_{L^\frac{2\gamma}{\gamma+1}}\aleq(\TS + h^A)^\frac{1}{2} \mbox{ for }\gamma\geq 2.
\end{align}

\item
If the numerical solutions are bounded, i.e., there exist positive constants $\overline{\vr} $ and $ \overline{\vu}$ such that 
$\vrh < \overline{\vr} \mbox{ and } \abs{\vuh}<\overline{\vu} \mbox{ uniformly for } h \rightarrow 0$. Then estimate~\eqref{ineq_error} still holds, where the convergence rate $A$ reads
$$A=\min\{1, 1+\eps, 2-\alpha, \alpha\}.$$ 
Moreover,
\begin{align}\label{con_b}
\norm{\vrh-\vr}_{L^2}+\norm{\vmh-\vm}_{L^2}\aleq(\TS + h^A)^\frac{1}{2}.
\end{align}
\end{itemize}
\end{theorem}
\begin{remark}\label{rem3}
Note that when the numerical solutions are bounded, the optimal rate of the relative energy is $1$ with the choice of $\alpha=1$ and any $\eps\geq 0$. In this case, the convergence rate of the numerical solution itself is $1/2$. Here, our assumption is weaker than that of~\cite{JovaRoh}.
\end{remark}
\begin{remark}\label{rem4}
In~\cite{LMSY1}, under the assumption of uniform boundedness of both the numerical solution and its total variation, a first-order convergence rate in terms of relative energy was achieved for the first time. However, in our results, only the boundedness of numerical solutions is required to obtain a first-order convergence rate with respect to relative energy. In the current method, the weak BV estimates produced by the Navier--Stokes type artificial diffusion replace the assumption on the total variation.
\end{remark}

\begin{proof}
Combining \eqref{Ener_Sta}, \eqref{cs1} with test function $\varphi=\Hc'(\vr) -\frac{1}{2}|\vu|^2$ and \eqref{cs2} with test function $\vvarphi=\vu $, we obtain
\begin{equation}\label{RE1}
\begin{aligned}
& \left[  \RE\{(\vrh,\vuh)|(\vr,\vu)\}   \right]_0^\tau + h^\alpha\int_0^\tau \intO{|\GradB\vuh|^2} \dt
\\& \leq \int_0^\tau \intOB{\vrh  \pdt  \frac{\abs{\vu}^2}2 + \vmh \cdot \Grad \frac{ \abs{\vu}^2}2 }\dt  \; +   R_\vr \left(\tau, \TS, h, \frac{|\vu|^2}{2} \right)
\\& - \int_0^\tau \intOB{\vrh   \pdt \Hc'(\vr) + \vmh \cdot \Grad \Hc'(\vr)} \dt  \;  -   R_\vr( \tau, \TS, h,\Hc'(\vr) )
\\& - \int_0^\tau \intOB{\vmh \cdot \pdt \vu + \frac{\vmh \otimes \vmh}{\vrh} : \Grad \vu + p(\vrh) \Div \vu} \dt  +  {R_{\vm} (\tau, \TS, h,-\vu)}
\\& + \int_0^\tau \intO{ \pdt \big( \vr \Hc'(\vr) - \Hc(\vr) \big) } \dt 
= e_S + R^E,
\end{aligned}
\end{equation}
where $e_S=    R_\vr \left(\tau, \TS, h,\frac{|\vu|^2}{2}\right) -
R_\vr(\tau, \TS, h, \Hc'(\vr) ) +  R_{\vm} (\tau, \TS, h,-\vu)$ is the consistency error and 
\begin{equation*}
\begin{aligned}
  R^E =  &- \intT{\intO{\big(\vrh (\avuh-\vu) \otimes (\avuh-\vu) : \Grad \vu\\
  &+
   \big(p(\vrh) - p'(\vr) (\vrh - \vr) -p(\vr) \big) \Div \vu} \big)} .
\end{aligned}
\end{equation*}
Thanks to H\"older's inequality we observe the following estimate
\begin{equation*}
\abs{R^E}  \aleq \norm{\Grad \vu}_{L^\infty( (0,T)\times \Omega)} \intT{ \RE\{(\vrh,\vuh)|(\vr,\vu)\}  }.
\end{equation*}
Revisiting the proof of the Lemma \ref{MAC_CS} in Appendix~\ref{App1} we have 
\begin{align*}
    \abs{e_S} \aleq \TS + h +h^{1+\eps} + h^{\frac{\alpha}{2}} + h^{1+\beta_D -\frac{\alpha}{2}} + h^{1+\beta_M -\frac{\alpha}{2}}.
\end{align*}
Substituting the above two estimates into \eqref{RE1} we obtain 
\begin{align*}
 & \RE\{(\vrh,\vuh)|(\vr,\vu)\}(\tau) + h^\alpha\int_0^\tau \intO{|\GradB\vuh|^2} \dt 
 \\ &
 \aleq  e_S   + \RE\{(\vrh,\vuh)|(\vr,\vu)\}(0)   + \intT{ \RE\{(\vrh,\vuh)|(\vr,\vu)\}  } 
 \\& 
 \aleq \TS + h +h^{1+\eps} +h^{\frac{\alpha}{2}} + h^{1+\beta_D -\frac{\alpha}{2}} + h^{1+\beta_M -\frac{\alpha}{2}} + \intT{ \RE\{(\vrh,\vuh)|(\vr,\vu)\}  } ,
\end{align*}
where we have also used the projection error  
$\RE\{(\vrh,\vuh)|(\vr,\vu)\}(0) \aleq  h^2$. 
Consequently, by Gronwall's lemma we have
\begin{multline*}
  \RE\{(\vrh,\vuh)|(\vr,\vu)\}(\tau) + h^\alpha\int_0^\tau \intO{|\GradB\vuh|^2} \dt 
 \aleq   \TS + h^A,\\
 A = \min\{ 1, 1+\eps, \frac{\alpha}{2}, 1+\beta_D -\frac{\alpha}{2}, 1+\beta_M -\frac{\alpha}{2}\}.    
\end{multline*}
Applying \eqref{eq:B1} we have 
\begin{align*}
  &\norm{\vrh-\vr}_{L^\gamma} + \norm{\vmh-\vm}_{L^{\frac{2\gamma}{\gamma+1}}}  \aleq   \RE\{(\vrh,\vuh)|(\vr,\vu)\}^{\frac{1}{\gamma}} \mbox{ for }\gamma\leq2,\\
  &\norm{\vrh-\vr}_{L^2} + \norm{\vmh-\vm}_{L^{\frac{2\gamma}{\gamma+1}}}  \aleq  \RE\{(\vrh,\vuh)|(\vr,\vu)\}^{\frac{1}{2}} \mbox{ for }\gamma\geq2;
\end{align*}
This confirms the first item of the theorem. We now proceed to establish the second item, under the assumption that \((\vrh, \vuh)\) is bounded.
In this case, the residual terms $e_i$ and $\hat e_i$ ($i=1,\dots,4$) arising in the consistency proof
(see Appendix~\ref{App1}, Lemma~\ref{MAC_CS}) can be estimated more sharply as follows:
\begin{equation*}
\begin{aligned}
&|e_1|\aleq h\norm{\vmh}_{L^\infty L^{\frac{2\gamma}{\gamma+1}}}\norm{\varphi}_{C^2}\aleq h,\\
&|e_2|\aleq h\norm{\vrh}_{L^2 L^{2}}\norm{\varphi}_{C^2}\norm{\GradB\vuh}_{L^2 L^2}\aleq h^{1-{\alpha}/{2}},\\
&|e_3|\aleq h\norm{\vrh}_{L^2 L^2}\norm{\varphi}_{C^1}\norm{\GradB\vuh}_{L^2 L^2}\aleq h^{1-{\alpha}/{2}},\\
&|e_4|\aleq h\norm{\vrh}_{L^2 L^2}\norm{\varphi}_{C^1}{ \norm{\GradB\vuh}_{L^2 L^2}}\aleq h^{1-{\alpha}/{2}},\\
&|\hat{e_1}|\aleq h\norm{\vrh|\overline{\vuh}|^2}_{L^1 L^1}\norm{\vvarphi}_{C^2}\aleq h,\\
&|\hat{e_2}|\aleq h\norm{\vrh\overline{\vuh}}_{L^2 L^2}\norm{\GradB\vuh}_{L^2 L^2}\norm{\vvarphi}_{C^2}\aleq h^{1-{\alpha}/{2}},\\
&|\hat{e_3}|\aleq h\norm{\vrh\overline{\vuh}}_{L^2 L^2}\norm{\GradB\vuh}_{L^2 L^2}\norm{\vvarphi}_{C^2}\aleq h^{1-{\alpha}/{2}},\\
&|\hat{e_4}|\aleq h\norm{\vrh\overline{\vuh}}_{L^2 L^2}{ \norm{\GradB\vuh}_{L^2 L^2}}\norm{\vvarphi}_{C^1}\aleq h^{1-{\alpha}/{2}}.\\
\end{aligned}    
\end{equation*}
Consequently, the consistency error $e_S$ satisfies
\begin{align*}
    \abs{e_S} &\aleq \TS + h +h^{1+\eps} + (h+h^\alpha) \norm{\GradB \vuh}_{L^2 L^2}\\
    &\aleq \TS + h+h^{1+\eps} +\frac{h^{2-\alpha}+h^\alpha}{\delta}+ \delta h^\alpha \norm{\GradB \vuh}_{L^2 L^2}^2,\quad \delta\in(0,1).
\end{align*}
Thus, we obtain
\begin{multline*}
  \RE\{(\vrh,\vuh)|(\vr,\vu)\}(\tau) + (1-\delta)h^\alpha\int_0^\tau \intO{|\GradB\vuh|^2} \dt 
 \aleq   \TS + h^A,\\
 A = \min\{ 1, 1+\eps, 2-\alpha, \alpha\} .    
\end{multline*}
Finally, based on \eqref{eq:B2} we have 
\begin{align*}
  \norm{\vrh-\vr}_{L^2} + \norm{\vmh-\vm}_{L^2}  \aleq   \RE\{(\vrh,\vuh)|(\vr,\vu)\}^{\frac{1}{2}}, 
\end{align*}
which completes the proof. 
\end{proof}

\section{Numerical experiments}\label{numerical_ex}
In this section, we conduct a two-dimensional example to illustrate the behavior of the MAC scheme \eqref{MAC_S} and provide numerical evidence consistent with the theoretical estimates. In particular, we consider the following error norms:
\begin{equation}\label{errors}
e_{\vr} \mbox{ = } \norm{\vrh - \vr_{ref}}_{L^\infty L^2 },~
e_{\vm} \mbox{ = } \norm{\vmh - \vm_{ref}}_{L^\infty L^2 }, ~ 
e_{\vu} \mbox{ = } \norm{\vuh - \vu_{ref}}_{L^2 L^2 }
\end{equation}
between the numerical solutions $(\vrh,\vuh)$ and a reference solution $(\vr_{ref},\vu_{ref})$, obtained by the MAC scheme \eqref{MAC_S} on a fine mesh with $h_{ref} \mbox{ = } 1/2048$. The computational domain is $\Omega \mbox{ = } [0,1]^2$ with periodic boundary conditions. 
In our numerical simulations, we set $\eps \mbox{ = }1.2$.
To solve the nonlinear system we use the fixed point iteration method, and in each subiteration we set $\Delta t \mbox{ = }CFL\frac{h}{|\vuh|_{max}+c_h}$ with $CFL\mbox{ = }0.4, c_h=\sqrt{\gamma\vrh^{\gamma-1}}$.

We consider the flow from an initial vortex centered at $\vx_c\mbox{ = } (0.5,0.5)$ with the radius $r_0 \mbox{ = }0.2:$
\begin{equation*}
    \vr_0\mbox{ = }1,~ \vu_0\mbox{ = }\frac{u_r(r)}{r}
    \left(\begin{aligned}
        \vx_2 - 0.5 \\
        0.5 - \vx_1 
    \end{aligned}\right),
    ~
    u_r(r)\mbox{=}\sqrt{\gamma}
    \begin{cases}
2r/r_0 & \mbox{ if } 0\leq r <r_0/2, \\
2(1-r/r_0) &\mbox{ if } r_0/2 \leq r <r_0, \\
0 &\mbox{ if } r \geq r_0,
\end{cases}
\end{equation*}
where $r\mbox{ = }\sqrt{(\vx_1-0.5)^2 + (\vx_2-0.5)^2}.$

We present the errors measured by the norms defined in \eqref{errors} for various $\gamma$ and $\alpha$, as shown in Table~\ref{ex_errors1} ($\alpha=0.5$), Table~\ref{ex_errors2} ($\alpha=0.8$) and Table~\ref{ex_errors3} ($\alpha=1.0$). From these tables, we observe that the optimal convergence rate is achieved when $\alpha=1$, which aligns with our analysis in Theorem~\ref{theo_EE}. Specifically, under the condition of bounded numerical solutions, the best theoretical convergence rate is obtained when $\alpha=1$. Although the theoretical optimal result is $1/2$, which is not as good as the numerical result, this remains the best theoretical result to date, as discussed in Remark~\ref{rem4}.

Moreover, in Fig.~\ref{ex_time} we present the time evolution of the fluid flow originating from the initial vortex. It is evident that the vortex dissipates gradually over time. This is further confirmed by Fig.~\ref{ex_energy}, which illustrates the energy dissipation.

\begin{table}[H]
\caption{Error of the MAC scheme ($T=0.1,~ \alpha=0.5$)}
\label{ex_errors1}
\centering
\renewcommand{\arraystretch}{1.4}
\fontsize{8pt}{6pt}\selectfont
\begin{tabular}{c p{0.9cm} p{1.1cm} p{0.8cm} p{1.1cm} p{0.8cm} p{1.1cm} p{0.8cm}}
    \toprule
    & $h$ &$e_{\vu}$ & rate & $e_{\vm}$ & rate & $e_{\vr}$ & rate \\
    \midrule
    \multirow{6}{*}{$\gamma = 1.4$} 
    & $1/64$ & 7.14e-03 & - & 1.13e-01 & - & 3.84e-02 & -\\
    & $1/128$ & 6.17e-03 & 0.21 & 9.28e-02 & 0.29 & 3.40e-02 & 0.18\\
    & $1/256$ & 4.89e-03 & 0.34 & 7.01e-02 & 0.40 & 2.77e-02 & 0.30\\
    & $1/512$ & 3.36e-03 & 0.54 & 4.63e-02 & 0.60 & 1.95e-02 & 0.51 \\
    \midrule
    \multirow{6}{*}{$\gamma = 1.67$} 
    & $1/64$ & 7.79e-03 & - & 1.23e-01 & - & 4.10e-02 & -\\
    & $1/128$ & 6.72e-03 & 0.21 & 1.01e-01 & 0.29 & 3.62e-02 & 0.18 \\
    & $1/256$ & 5.31e-03 & 0.34 & 7.57e-02 & 0.41 & 2.94e-02 & 0.30 \\
    & $1/512$ & 3.64e-03 & 0.54 & 4.98e-02 & 0.60 & 2.06e-02 & 0.51 \\
    \midrule
    \multirow{6}{*}{$\gamma = 2$} 
    & $1/64$ & 8.52e-03 & - & 1.34e-01 & - & 4.38e-02 & -\\
    & $1/128$ & 7.34e-03 & 0.22 & 1.09e-01 & 0.30 & 3.85e-02 & 0.18\\
    & $1/256$ & 5.79e-03 & 0.34 & 8.18e-02 & 0.41 & 3.12e-02 & 0.30\\
    & $1/512$ & 3.96e-03 & 0.55 & 5.36e-02 & 0.61 & 2.19e-02 & 0.52 \\
    \bottomrule
\end{tabular}
\end{table}
\begin{table}[H]
\centering
\caption{Error of the MAC scheme ($T=0.1,~ \alpha=0.8$)}
\label{ex_errors2}
\renewcommand{\arraystretch}{1.4}
\fontsize{8pt}{6pt}\selectfont
\begin{tabular}{c p{0.9cm} p{1.1cm} p{0.8cm} p{1.1cm} p{0.8cm} p{1.1cm} p{0.8cm}}
    \toprule
        & $h$ &$e_{\vu}$ & rate & $e_{\vm}$ & rate & $e_{\vr}$ & rate \\
        \midrule
        \multirow{6}{*}{$\gamma = 1.4$} 
        & $1/64$ & 8.94e-03 & - & 1.20e-01 & - & 5.59e-02 & -\\
        & $1/128$ & 6.50e-03 & 0.46 & 8.76e-02 & 0.45 & 3.99e-02 & 0.49\\
        & $1/256$ & 4.26e-03 & 0.61 & 5.79e-02 & 0.60 & 2.49e-02 & 0.68\\
        & $1/512$ & 2.42e-03 & 0.82 & 3.30e-02 & 0.81 & 1.33e-02 & 0.90 \\
        \midrule
        \multirow{6}{*}{$\gamma = 1.67$} 
        & $1/64$ & 9.67e-03 & - & 1.28e-01 & - & 5.83e-02 & -\\
        & $1/128$ & 7.03e-03 & 0.46 & 9.37e-02 & 0.45 & 4.13e-02 & 0.50 \\
        & $1/256$ & 4.61e-03 & 0.61 & 6.21e-02 & 0.59 & 2.57e-02 & 0.68 \\
        & $1/512$ & 2.63e-03 & 0.81 & 3.55e-02 & 0.81 & 1.37e-02 & 0.91 \\
        \midrule
        \multirow{6}{*}{$\gamma = 2$} 
        & $1/64$ & 1.05e-02 & - & 1.37e-01 & - & 6.10e-02 & -\\
        & $1/128$ & 7.63e-03 & 0.46 & 1.01e-01 & 0.45 & 4.31e-02 & 0.50\\
        & $1/256$ & 5.02e-03 & 0.61 & 6.70e-02 & 0.59 & 2.67e-02 & 0.69\\
        & $1/512$ & 2.87e-03 & 0.80 & 3.85e-02 & 0.80 & 1.41e-02 & 0.92 \\
        \bottomrule
\end{tabular}
\end{table}
\vspace{-1pt}
\begin{table}[H]
\centering
\caption{Error of the MAC scheme ($T=0.1,~ \alpha=1$)}
\label{ex_errors3}
\renewcommand{\arraystretch}{1.4}
\fontsize{8pt}{6pt}\selectfont
\begin{tabular}{c p{0.9cm} p{1.1cm} p{0.8cm} p{1.1cm} p{0.8cm} p{1.1cm} p{0.8cm}}
    \toprule
        & $h$ &$e_{\vu}$ & rate & $e_{\vm}$ & rate & $e_{\vr}$ & rate \\
        \midrule
        \multirow{6}{*}{$\gamma = 1.4$} 
        & $1/64$ & 6.79e-03 & - & 9.27e-02 & - & 4.36e-02 & -\\
        & $1/128$ & 4.32e-03 & 0.65 & 5.94e-02 & 0.64 & 2.59e-02 & 0.75\\
        & $1/256$ & 2.48e-03 & 0.80 & 3.40e-02 & 0.80 & 1.38e-02 & 0.91\\
        & $1/512$ & 1.25e-03 & 0.99 & 1.69e-02 & 1.00 & 6.41e-03 & 1.10 \\
        \midrule
        \multirow{6}{*}{$\gamma = 1.67$} 
        & $1/64$ & 3.02e-02 & - & 9.97e-02 & - & 4.50e-02 & -\\
        & $1/128$ & 1.93e-02 & 0.65 & 6.42e-02 & 0.64 & 2.67e-02 & 0.76 \\
        & $1/256$ & 1.11e-02 & 0.79 & 3.69e-02 & 0.80 & 1.41e-02 & 0.92 \\
        & $1/512$ & 5.64e-03 & 0.98 & 1.84e-02 & 1.00 & 6.58e-03 & 1.10 \\
        \midrule
        \multirow{6}{*}{$\gamma = 2$} 
        & $1/64$ & 8.07e-03 & - & 1.08e-01 & - & 4.68e-02 & -\\
        & $1/128$ & 5.19e-03 & 0.64 & 6.98e-02 & 0.63 & 2.77e-02 & 0.76\\
        & $1/256$ & 3.02e-03 & 0.78 & 4.03e-02 & 0.79 & 1.47e-02 & 0.92\\
        & $1/512$ & 1.54e-03 & 0.97 & 2.02e-02 & 1.00 & 6.83e-03 & 1.10 \\
        \bottomrule
\end{tabular}
\end{table}
\begin{figure}[H]
    \centering
    \begin{subfigure}[b]{0.32\textwidth}
        \includegraphics[width=0.8\textwidth, height=0.2\textheight,  keepaspectratio]{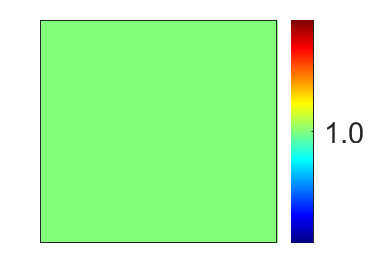}
    \end{subfigure}
    \hspace{-30pt}
    \begin{subfigure}[b]{0.32\textwidth}
        \includegraphics[width=0.8\textwidth, height=0.2\textheight,  keepaspectratio]{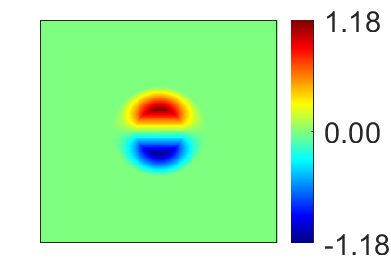}
    \end{subfigure}
    \hspace{-30pt}
    \begin{subfigure}[b]{0.32\textwidth}
        \includegraphics[width=0.8\textwidth, height=0.2\textheight,  keepaspectratio]{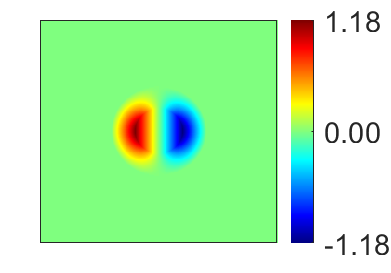}
    \end{subfigure}
     \centering
    \hspace{0.025\textwidth}
    \begin{subfigure}[b]{0.32\textwidth}
        \includegraphics[width=0.8\textwidth, height=0.2\textheight,  keepaspectratio]{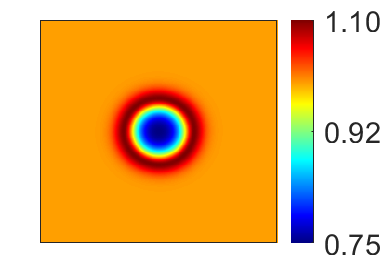}
    \end{subfigure}
    \hspace{-30pt}
    \begin{subfigure}[b]{0.32\textwidth}
        \includegraphics[width=0.8\textwidth, height=0.2\textheight,  keepaspectratio]{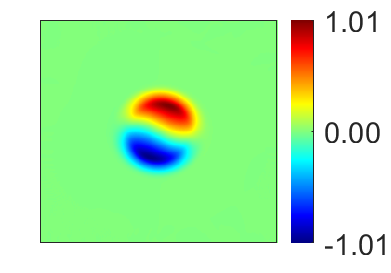}
    \end{subfigure}
    \hspace{-30pt}
    \begin{subfigure}[b]{0.32\textwidth}
        \includegraphics[width=0.8\textwidth, height=0.2\textheight,  keepaspectratio]{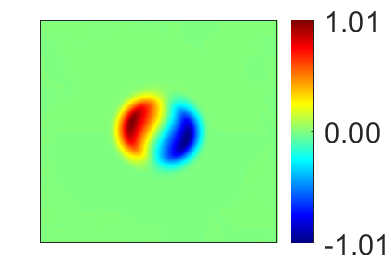}
    \end{subfigure}
    \centering
    \hspace{0.025\textwidth}
    \begin{subfigure}[b]{0.32\textwidth}
        \includegraphics[width=0.8\textwidth, height=0.2\textheight,  keepaspectratio]{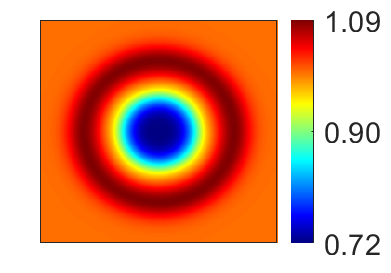}
    \end{subfigure}
    \hspace{-30pt}
    \begin{subfigure}[b]{0.32\textwidth}
        \includegraphics[width=0.8\textwidth, height=0.2\textheight,  keepaspectratio]{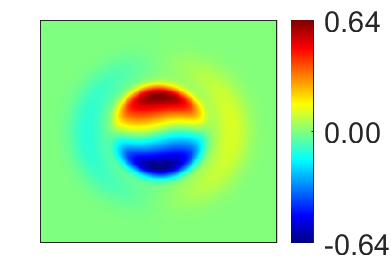}
    \end{subfigure}
    \hspace{-30pt}
    \begin{subfigure}[b]{0.32\textwidth}
        \includegraphics[width=0.8\textwidth, height=0.2\textheight,  keepaspectratio]{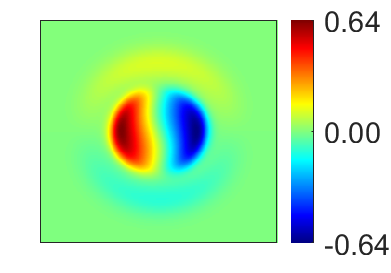}
    \end{subfigure}
 \caption{Time evolution of the flow with $\gamma = 1.4,~ \alpha = 0.8$. From top to down are $t\mbox{ = } 0,~0.05,~ 0.2$. From left to right are $\vr,~\vu_1$, and $\vu_2$.}
 \label{ex_time}
\end{figure}
\begin{figure}[H]
    \hfill
    \centering
    \begin{subfigure}[b]{\textwidth}
        \includegraphics[width=0.8\textwidth]{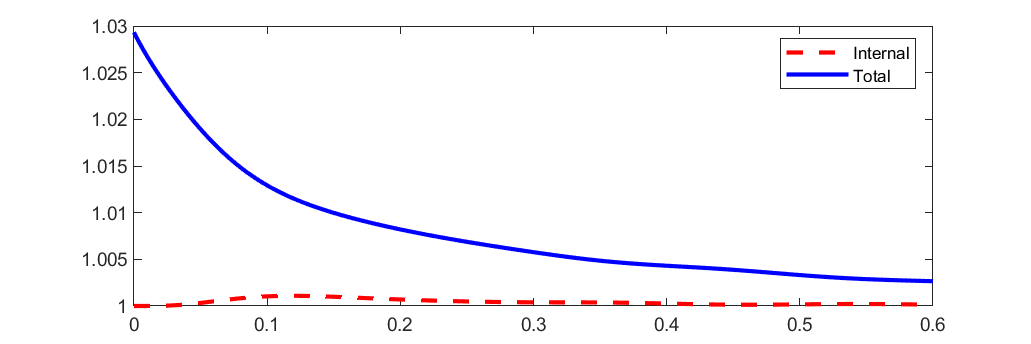}
    \end{subfigure}
    \hfill
    \hfill
    \begin{subfigure}[b]{\textwidth}
        \includegraphics[width=0.8\textwidth]{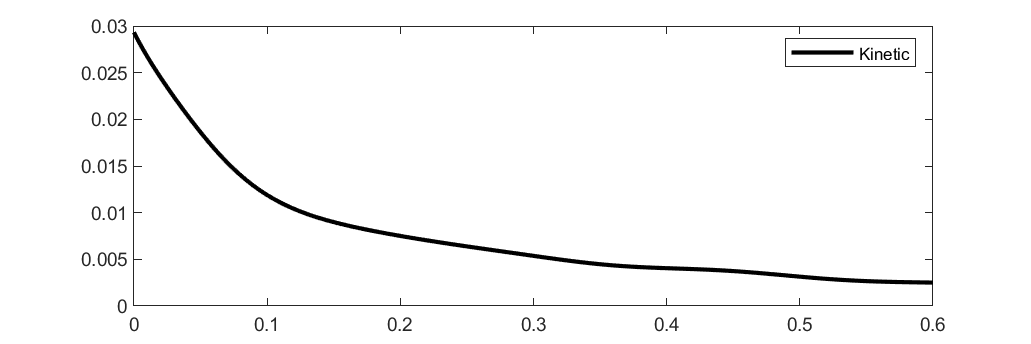}
    \end{subfigure}
    \hfill
 \caption{Time evolution of the total, internal, and kinetic energy with $\gamma = 1.4,~ \alpha = 0.8$ up to $T=0.6$.}
 \label{ex_energy}
\end{figure}
\section{Conclusion}
In this paper, we analyze the convergence and convergence rate of a MAC scheme for the barotropic Euler system. We prove the unconditional convergence of the MAC scheme. Under the additional assumption that the numerical sequence is uniformly bounded, we obtain a convergence rate of $1/2$ for the numerical solution. Previous studies~\cite{LMSY1,LeVeque} have achieved the same rate under stronger assumptions.
A key point of these results is the weak-BV estimate on the velocity, which is provided by the Navier--Stokes type artificial diffusion term.

The unconditional convergence of the MAC scheme is established in Section~\ref{Con_MAC}, where the first main result is presented in Theorem~\ref{theo_Con}. We prove this result by employing a pre-built convergence theorem in the spirit of Lax and Rychtmyer \cite{LaxR}, based on the stability and consistency of the scheme.

Then, utilizing again stability and consistency, we derive estimates for the relative energy functional, which allow us to obtain the theoretical convergence rate of the MAC scheme, see the first result of Theorem \ref{theo_EE}. Furthermore, assuming that the numerical solutions are uniformly bounded, we derive the optimal convergence rate of 1/2, see the second result of Theorem \ref{theo_EE}.

Finally, in the experiment, we consider two-dimensional vortex flow, where we observe a convergence rate of up to 1. It is worth noting that, although our theoretical results only show an optimal convergence rate of 1/2. To the best of our knowledge, this is comparable to the best available relative-energy-based estimate under similar assumptions.

\section*{Acknowledgment}
J. Zhao was supported by the Beijing Natural Science Foundation under Grant No. JR25003 and by the National Natural Science Foundation of China under Grant No. 12301520; B. She and C. Wang were supported by the National Natural Science Foundation of China under Grant No. 12571433.

\appendix

\section{Proof of Lemma~\ref{MAC_CS}}\label{App1}
\begin{proof}
The consistency proof for the MAC scheme closely follows the approach used in our previous work on the compressible Navier--Stokes equations~\cite[Chapter 14.3]{FLMS_book} and~\cite{FLS_IEE}. In this work, we outline the main steps of the proof and emphasize the key differences arising from the absence of estimates on the velocity gradient. The strategy proceeds as follows:
\begin{itemize}
    \item We set $\phi_h=\Pi_Q\varphi$ and $\Phi_h=\Pi_{\mathcal E}\vvarphi$ in \eqref{MAC_SD} and \eqref{MAC_SM}, respectively. The discrete projection $\Pi_{\mathcal E}$ is defined by:
    $$\Piv \bm{f}= (\Piv^{(1)}f_1,\dots,\Piv^{(d)}f_d),\quad \Pivi f = \sum_{\sigma \in \facei}  (\Pivi f)_{\sigma} 1_{D_\sigma},
\quad  (\Pivi f)_{\sigma} =   \frac{ 1}{|\sigma|} \intSh{ f },$$
    \item Reformulate the resulting equations from the previous step by replacing $\Piq \varphi$ and ${\Piv \vvarphi}$ respectively with $\varphi$ and $\vvarphi$. 
    \item Estimate all residual terms in the following steps. 
\end{itemize}


\noindent{\bf The time derivative terms:} 
\begin{equation*}
\begin{aligned}
&\int_0^T \intO{\left( D_t \vrh \Piq\varphi + \vrh\partial_t\varphi\right)} \dt + \intO{\vrh^0\varphi(0,\cdot)}\\
&\aleq \Delta t||\varphi||_{C^2}||\vrh||_{L^1L^1} + \Delta t||\varphi||_{C^1}||\vrh^0||_{L^1} \aleq \Delta t,    
\end{aligned}
\end{equation*}
and
\begin{equation*}
\begin{aligned}
&\int_0^T \intOB{ D_t\vmh\cdot\overline{\Piv\vvarphi} + \vmh\cdot\partial_t\vvarphi } \dt + \intO{\vmh^0 \cdot \vvarphi(0)}\\
&\aleq \Delta t \left( ||\vmh||_{L^1L^1}||\vvarphi||_{C^2} + ||\vmh^0||_{L^1}||\vvarphi||_{C^1} \right) \aleq \Delta t.
\end{aligned}
\end{equation*}
\noindent{\bf The pressure term:} 
\begin{equation*}
\int_0^T \intO{p(\vrh)\Divh\Piv\vvarphi} \dt = \int_0^T \intO{p(\vrh)\Div\vvarphi} \dt.
\end{equation*}

\noindent{\bf The artificial diffusion terms:} 
\begin{equation*}
\begin{aligned}
&h^\eps\int_0^T \intO{\jump{\vrh}\jump{\varphi}} \dt = -h^{\eps+1}\int_0^T \intO{\vrh\Delta_h\varphi} \dt\\
&\aleq h^{\eps+1} ||\vrh||_{L^\infty L^\gamma}||\varphi||_{C^2}\aleq h^{\eps+1},
\end{aligned}
\end{equation*}
\begin{equation*}
\begin{aligned}
&h^\eps\int_0^T \intO{\jump{\vmh}\cdot\jump{\vvarphi}} \dt = -h^{\eps+1}\int_0^T \intO{\vmh\cdot\Delta_h\vvarphi} \dt\\
&\aleq h^{\eps+1} ||\vmh||_{L^\infty L^{\frac{2\gamma}{\gamma+1}}}||\vvarphi||_{C^2}\aleq h^{\eps+1},
\end{aligned}
\end{equation*}
$\GradD, \Delta_h$ denotes the gradient and Laplace operators for any $r_h\in\Qh$, defined as follows
\begin{equation*}
\begin{aligned}
&\GradD r_h(\bfx) = \left( \eth_{{\cal D}_1} r_h, \ldots,  \eth_{{\cal D}_d} r_h \right)(\bfx),  \\
\Delta_h r_h = &\sum_{i=1}^d\Delta_h^{(i)} r_h,\quad \Delta_h^{(i)} r_h = \pdmeshi(\pdduali r_h).  
\end{aligned}
\end{equation*}
where
\begin{equation*}
 \pdedgei r_h (\bfx) = \sum_{\sigma \in \edgesi}1_{D_\sigma} (\pdedgei r_h)_{\sigma} ,  \quad  \ (\pdedgei r_h)_{\sigma}  = \frac{r_{L} - r_{K}}{h}, \quad \sigma=K|L\in \edgesi.
\end{equation*}


\noindent{\bf The convective terms:}
Applying Lemma 8.1 of \cite{FLMS_book}, we obtain
\begin{equation*}
\begin{aligned}
\intE{{\rm Up}[\vrh, \vuh]\jump{\Piq\varphi}} &= 
\intO{\vrh\vuh\cdot \GradD(\Piq \varphi)} \\
&+\frac{h}{2}\sum_{i=1}^d\sum_{K\in\mesh}\int_K \vrh\Delta_h^{(i)}(\Piq\varphi)\overline{|\uih|} dx\\
&+\frac{h}{2}\sum_{i=1}^d\sum_{K\in\mesh}\int_K \vrh \overline{\pdduali(\Piq\varphi)}\pdmeshi|\uih| dx,
\end{aligned}
\end{equation*}
and
\begin{equation*}
\begin{aligned}
\intE{{\rm Up}[(\vrh\overline{\uih}), \vuh]\jump{\overline{\Pivi\varphi_i}}} &=
\intO{\vrh\overline{\uih}\vuh\cdot \GradD(\Piq\Pivi\varphi_i)}\\
&+\frac{h}{2}\sum_{j=1}^d\sum_{K\in\mesh}\int_K (\vrh\overline{\uih})\Delta_h^{(j)}(\Piq\Pivi\varphi_i)\overline{|\ujh|} dx\\
&+\frac{h}{2}\sum_{j=1}^d\sum_{K\in\mesh}\int_K (\vrh\overline{\uih})\overline{\pddualj(\Piq\Pivi\varphi_i)}\pdmeshj|\ujh| dx.
\end{aligned}
\end{equation*}

Based on the above two equations, H\"older's inequality and the uniform bounds in Corollary \ref{cor_unib}, we can obtain

\begin{equation*}
\begin{aligned}
&\int_0^{ t^{n+1}}\intE{{\rm Up}[\vrh, \vuh]\jump{\Piq\varphi}}\dt - \int_0^{ t^{n+1}}\intO{\vrh\overline{\vuh} \cdot \Grad\varphi}\dt 
\\
&
=\int_0^{t^{n+1}} \intO{\vrh\overline{\vuh}\cdot\left( \GradD(\Piq \varphi)-\Grad\varphi \right)} \dt\\
&+\frac{h}{2}\sum_{i=1}^d\int_0^{t^{n+1}} \sum_{K\in\mesh}\intK{\vrh\Delta_h^{(i)}(\Piq\varphi)|\overline{\uih}|} \dt\\
&+\frac{h}{2}\sum_{i=1}^d\int_0^{t^{n+1}} \sum_{K\in\mesh}\intK{\vrh \overline{\pdduali(\Piq\varphi)}\pdmeshi|\uih|} \dt\\
&+\int_0^{t^{n+1}} \intO{\vrh(\vuh - \overline{\vuh})\cdot\GradD(\Piq \varphi)} \dt =:\sum_{i=1}^4e_i,
\end{aligned}
\end{equation*}
and the terms $e_i, i=1,2,3,4$ are controlled as follows
\begin{equation*}
\begin{aligned}
&|e_1|\aleq h\norm{\vmh}_{L^\infty L^{\frac{2\gamma}{\gamma+1}}}\norm{\varphi}_{C^2}\aleq h,\\
&|e_2|\aleq h\norm{\vrh}_{L^2 L^{2}}\norm{\varphi}_{C^2}\norm{\GradB\vuh}_{L^2 L^2}\aleq h^{1+\beta_D-{\alpha}/{2}},\\
&|e_3|\aleq h\norm{\vrh}_{L^2 L^2}\norm{\varphi}_{C^1}\norm{\GradB\vuh}_{L^2 L^2}\aleq h^{1+\beta_D-{\alpha}/{2}},\\
&|e_4|\aleq h\norm{\vrh}_{L^2 L^2}\norm{\varphi}_{C^1}{ \norm{\GradB\vuh}_{L^2 L^2}}\aleq h^{1+\beta_D-{\alpha}/{2}}.
\end{aligned}    
\end{equation*}

\begin{equation*}
\begin{aligned}
\int_0^T \intE{&{\rm Up}[\vmh, \vuh]\cdot\jump{\overline{\Piv\vvarphi}} } \dt -\int_0^T\intO{(\vrh\overline{\vuh}\otimes\vuh):\Grad\vvarphi}\dt=\\
&+\sum_{i=1}^d\int_0^T \intO{\vrh\overline{\uih}\vuh\cdot\left( \GradD(\Piq\Pivi\varphi_i)-\Grad\varphi_i \right)} \dt\\
&+\sum_{i=1}^d\sum_{j=1}^d\int_0^T \frac{h}{2}\sum_{K\in\mesh}\int_K (\vrh\overline{\uih})\Delta_h^{(j)}(\Piq\Pivi\varphi_i)|\overline{\ujh}| \dx \dt\\
&+\sum_{i=1}^d\sum_{j=1}^d\int_0^T \frac{h}{2}\sum_{K\in\mesh}\int_K (\vrh\overline{\uih})\overline{\pddualj(\Piq\Pivi\varphi_i)}\pdmeshj|\ujh| \dx \dt\\
&+\int_0^T \intO{\vrh\overline{\vuh}\otimes(\vuh - \overline{\vuh}):\Grad\vvarphi} \dt =:
\sum_{i=1}^4 \hat{e}_i,
\end{aligned}
\end{equation*}
The terms $\hat e_i$, $i=1,2,3$, are controlled as follows:
\begin{equation*}
\begin{aligned}
&|\hat{e_1}|\aleq h\norm{\vrh|\overline{\vuh}|^2}_{L^1 L^1}\norm{\vvarphi}_{C^2}\aleq h,\\
&|\hat{e_2}|\aleq h\norm{\vrh\overline{\vuh}}_{L^2 L^2}\norm{\GradB\vuh}_{L^2 L^2}\norm{\vvarphi}_{C^2}\aleq h^{1+\beta_M-{\alpha}/{2}},\\
&|\hat{e_3}|\aleq h\norm{\vrh\overline{\vuh}}_{L^2 L^2}\norm{\GradB\vuh}_{L^2 L^2}\norm{\vvarphi}_{C^2}\aleq h^{1+\beta_M-{\alpha}/{2}},\\
&|\hat{e_4}|\aleq h\norm{\vrh\overline{\vuh}}_{L^2 L^2}{ \norm{\GradB\vuh}_{L^2 L^2}}\norm{\vvarphi}_{C^1}\aleq h^{1+\beta_M-{\alpha}/{2}}.\\
\end{aligned}    
\end{equation*}

\noindent{\bf The viscosity term:} 
\begin{equation*}
\begin{aligned}
h^\alpha\int_0^T \intO{\GradB \vuh : \GradB \Piv \vvarphi} \dt \aleq h^\alpha{ \norm{\GradB\vuh}_{L^2 L^2}} \norm{\GradB \Piv \vvarphi}_{L^2 L^2} \aleq h^{{\alpha}/{2}}.
\end{aligned} 
\end{equation*}

Collecting the above estimates finishes the proof.
\end{proof}

\section{Relative energy norm}
In this section, we recall an inequality that controls the error in the conservative variables in terms of the relative energy, following \cite[Lemma~C.1]{FLS_IEE}.

\begin{lemma}\label{lem:B1}
Let $\gamma>1$ and $(r,\bm{U})$ satisfy
\[
\underline r=\frac{1}{2}\min_{(t,x)\in (0,T)\times\Omega} r,
\qquad
\overline r=2\max_{(t,x)\in (0,T)\times\Omega} r,
\qquad
\overline u=\max_{(t,x)\in (0,T)\times\Omega}|\bm{U}|
\]
for some positive constants $\overline u$, $\underline r$, $\overline r$.

\begin{itemize}
\item If $\vr>0$ and $\int_{\mathbb{T}^d}\vr^\gamma\,\mathrm{d}x \le E_0$
then the following estimates hold:
\begin{subequations}\label{eq:B1}
\begin{align}
\|\vr-r\|_{L^\gamma}
+
\|\vm-\mathbf{M}\|_{L^{\frac{2\gamma}{\gamma+1}}}
&\lesssim
\bigl(\RE(\vr,\vu\mid r,\bm{U})\bigr)^{\frac12}\notag\\
&+
\bigl(\RE(\vr,\vu\mid r,\bm{U})\bigr)^{\frac1\gamma},
\qquad \text{for } \gamma\le 2,\\
\|\vr-r\|_{L^2}
+
\|\vm-\mathbf{M}\|_{L^{\frac{2\gamma}{\gamma+1}}}
&\lesssim
\bigl(\RE(\vr,\vu \mid r,\bm{U})\bigr)^{\frac12},
\qquad \text{for } \gamma\ge 2,
\end{align}
\end{subequations}
where $\vm=\vr \vu$ and $\mathbf{M}=r\bm{U}$.

\item In addition, let $\vr<\overline\vr$. Then
\begin{equation}\label{eq:B2}
\|\vr-r\|_{L^2}
+
\|\vm-\mathbf{M}\|_{L^2}
\lesssim
\bigl(\RE(\vr,\vu\mid r,\bm{U})\bigr)^{\frac12}.
\end{equation}
\end{itemize}
\end{lemma}

\end{document}